\documentclass[11pt]{amsart}
\usepackage[T1]{fontenc}
\usepackage{lmodern}
\usepackage[a4paper,margin=1in]{geometry}
\usepackage{amssymb,mathtools}
\usepackage{microtype}
\usepackage[colorlinks=true,linkcolor=blue,citecolor=blue,urlcolor=blue]{hyperref}

\numberwithin{equation}{section}
\newtheorem{theorem}{Theorem}[section]
\newtheorem{lemma}[theorem]{Lemma}
\newtheorem{proposition}[theorem]{Proposition}

\theoremstyle{remark}
\newtheorem{remark}[theorem]{Remark}

\newcommand{\CP}{\mathbb{CP}}
\newcommand{\C}{\mathbb{C}}
\newcommand{\R}{\mathbb{R}}
\newcommand{\TCl}{T_{\mathrm{Cl}}}
\newcommand{\ip}[2]{\langle #1,#2\rangle}
\DeclareMathOperator{\Hess}{Hess}
\DeclareMathOperator{\tr}{tr}
\DeclareMathOperator{\SU}{SU}
\DeclareMathOperator{\PU}{PU}
\DeclareMathOperator{\SO}{SO}
\renewcommand{\Re}{\operatorname{Re}}
\renewcommand{\Im}{\operatorname{Im}}
\allowdisplaybreaks[2]

\title[Minimal Lagrangian tori in $\CP^2$]{Uniqueness of embedded minimal Lagrangian tori in $\CP^2$}

\author{Yong Luo}
\address{Mathematical Science Research Center, Chongqing University of
Technology, Chongqing 400054, P.R. China}
\email{yongluo-math@cqut.edu.cn}

\author{Hui Ma}
\address{Department of Mathematical Sciences, Tsinghua University, Beijing 100084,
P.R. China}
\email{ma-h@tsinghua.edu.cn}

\author{Jiabin Yin}
\address{School of Mathematics and Statistics, Xinyang Normal University,
Xinyang 464000, P.R. China}
\email{jiabinyin@126.com}

\subjclass[2020]{Primary 53C42; Secondary 53D12, 53C24}
\keywords{Minimal Lagrangian tori, complex projective plane,
Clifford torus, two-point function}
\hypersetup{
 pdftitle={Uniqueness of embedded minimal Lagrangian tori in CP2},
 pdfauthor={Yong Luo, Hui Ma, Jiabin Yin}
}

\begin{document}

\begin{abstract}
We prove that any embedded minimal Lagrangian torus in the complex
projective plane is congruent to the Clifford torus.
This settles the orientable case of Oh's classification conjecture.
The proof uses a new two-point function constructed from
Hermitian products of normalized horizontal lifts.
\end{abstract}

\maketitle

\section{Introduction}

The rigidity of embedded minimal surfaces is a central theme in
differential geometry. A landmark result is Brendle's proof of the
Lawson conjecture for embedded minimal tori in the round three-sphere.
Using a two-point argument, he proved the following theorem.

\begin{theorem}[Brendle \cite{Brendle2013}]
Let $\Sigma$ be an embedded minimal torus in $\mathbb{S}^3$. Then $\Sigma$ is congruent to the Clifford torus $\mathbb{S}^1(\frac{\sqrt{2}}{2})\times\mathbb{S}^1(\frac{\sqrt{2}}{2})$.
\end{theorem}

In this paper, we establish the Lagrangian analogue of this theorem
in $\CP^2$, proving that every embedded minimal Lagrangian torus
is congruent to the Clifford torus.

Minimal Lagrangian submanifolds connect minimal submanifold theory
with symplectic geometry. A submanifold is Lagrangian if it has
half the real dimension of the ambient symplectic manifold and
the symplectic form vanishes on its tangent spaces.
Their study is closely related to variational problems for volume
\cite{Oh1990,Oh1993,SchoenWolfson2001}
and calibrated geometry \cite{HarveyLawson1982}.

A basic example in $\CP^2$ is the Clifford torus
\[
T_{\mathrm{Cl}}
=\bigl\{[z_0:z_1:z_2]\in\CP^2:
         |z_0|=|z_1|=|z_2|\bigr\},
\]
which is flat, homogeneous, and minimal Lagrangian.

Immersed minimal Lagrangian tori in $\CP^2$ form a rich
class of examples, studied through equivariant constructions
\cite{CastroUrbano1994}, integrable systems
\cite{Sharipov1991,MaMa2005,McIntosh2003,
CarberryMcIntosh2004,Mironov2010}, and special Lagrangian
cones \cite{Joyce2008,Haskins2004}.
Within the integrable systems approach, Dorfmeister and Ma
developed a loop group (DPW) method for constructing minimal
Lagrangian surfaces in $\CP^2$;
see \cite{DorfmeisterMa2021,DorfmeisterMa2025}
and the references therein.

The classification of embedded examples is a distinct
rigidity problem.
Oh \cite[p.~482]{Oh1994} conjectured that, up to congruence,
the standard real projective plane and the Clifford torus
are the only compact embedded minimal Lagrangian surfaces
in $\CP^2$.
Ma presented the torus case of this conjecture as the
Lagrangian version of the Lawson conjecture in her
2012 Leuven lecture \cite{Ma2012Leuven}.

Every closed connected orientable Lagrangian surface $L$
embedded in $\CP^2$ has genus one.
Indeed, the Lagrangian condition makes $L$ null-homologous,
so embeddedness implies that its normal Euler number vanishes
\cite[Lemma~4.6]{EschenburgGuadalupeTribuzy1985}.
The complex structure identifies the tangent and normal bundles
with opposite orientations, so $\chi(L)=0$.
Thus the following theorem settles the orientable case
of Oh's conjecture.

\begin{theorem}\label{thm:main}
Let $f:T^2\to\mathbb{CP}^2$ be a smooth minimal Lagrangian embedding.
Then there is $A\in\mathrm{PU}(3)$ such that
$f(T^2)=A(T_{\mathrm{Cl}})$.
\end{theorem}

Our proof uses a new scalar two-point function constructed
from normalized horizontal lifts. It encodes the Lagrangian
geometry in real codimension two without requiring a choice
of unit normal.
Two-point methods have been used in geometric flows
\cite{Huisken1998,AndrewsBryan2011,Andrews2012}
and in symmetry and classification results for embedded
tori in $\mathbb{S}^3$ \cite{Brendle2013,AndrewsLi2015,Brendle2014}.
For a survey of the method and its applications, see
\cite{BrendleTwoPoint2014}.

We equip $\CP^2$ with the Fubini--Study metric of constant
holomorphic sectional curvature $4$.
Write $L=T^2$ with the metric induced by $f$, and choose
an orientation on $L$.
We normalize a local horizontal lift $F$ of $f$ under
the Hopf projection $\mathbb{S}^5\to\CP^2$ by
$\det\nolimits_{\mathbb C}(F,E_1,E_2)=1$,
where $E_j=dF(e_j)$ and $(e_1,e_2)$ is a local positively
oriented orthonormal tangent frame.
This normalization determines the local lift up to
multiplication by a cube root of unity.

For $x,y\in L$, choose normalized local horizontal lifts
$F_1,F_2$ of $f$ near $x,y$, respectively, and set
$a=\ip{F_1(x)}{F_2(y)}$. Define
\begin{equation}\label{eq:intro-functions}
 D(x,y)=1-|a|^2,\qquad
 G(x,y)=1-3|a|^2+2\Re(a^3).
\end{equation}
Changing the normalized lifts multiplies $a$ by a cube root
of unity. Thus $D$ and $G$ are independent of these choices
and define smooth functions on $L\times L$.
Geometrically, $D=|F_2(y)-aF_1(x)|^2$ is the squared distance
from $F_2(y)$ to the complex line $\C F_1(x)$.
The identity $G=\prod_{j=0}^{2}(1+2\Re(\zeta^j a))$,
where $\zeta=e^{2\pi i/3}$, expresses $G$ as the product
of three shifted heights.
Embeddedness gives $D>0$ off the diagonal.
The key estimate is the sharp inequality
\[
 \Phi(x,y):=\frac{G(x,y)}{D(x,y)^2}\geq\frac34,
 \qquad x\ne y.
\]
To prove this estimate, we establish a critical-point
Hessian identity for functions of the Hermitian product
of normalized lifts. A first application of this identity
to the real part of the Hermitian product, combined with
embeddedness, gives $G>0$ off the diagonal.
A diagonal expansion shows that $D^2$ and $G$ both vanish
to fourth order, and that their quotient $\Phi$ extends
continuously to the diagonal with value $3/4$.
If its minimum were $k<3/4$, a minimizing pair
would lie off the diagonal, where $Z_k=G-kD^2$ attains
its minimum value zero.
A second application of the same Hessian identity to $Z_k$
at this pair gives a strictly negative sum of two Hessian
evaluations, contradicting the minimum condition.

The sharp inequality and the diagonal expansion of $\Phi$,
together with the Gauss equation, give nonnegative Gauss
curvature on $L$. Since $L$ is a torus, the Gauss--Bonnet
theorem forces its induced metric to be flat.
The classical classification of flat minimal Lagrangian
tori then shows that $f(L)$ is congruent to the Clifford torus.

The rest of this paper is organized as follows. Section~\ref{sec:Horizontal_lifts} constructs the normalized three-sheeted covering and derives the basic tensor identities and the diagonal expansion of the quotient. Section~\ref{sec:two-point} proves the sharp inequality by a two-point Hessian argument. Section~\ref{sec:rigidity} completes the proof of Theorem~\ref{thm:main}.

\section{Horizontal lifts and the cubic tensor}
\label{sec:Horizontal_lifts}
We recall the cubic tensor identities, construct the normalized
three-sheeted covering, and derive the diagonal expansion of $G/D^2$.
This expansion is used both in the two-point argument and in the final
curvature estimate.

We use the Hermitian product
\[
 \ip{z}{w}=\sum_{\alpha=0}^{2}\overline{z_\alpha}w_\alpha
\]
on $\C^3$, which is conjugate-linear in the first variable. Its real
part is the Euclidean metric. We write $g$ for the Fubini--Study metric
of constant holomorphic sectional curvature $4$ and for the induced
metric on $L$. We denote the complex structure by $J$ and the
K\"ahler form by $\omega$, with $\omega(X,Y)=g(JX,Y)$.
We use the convention $\Delta=\operatorname{div}\nabla$.

Let $f:L\to\CP^2$ be a minimal Lagrangian immersion of an oriented
surface, with second fundamental form $B$. Define the cubic tensor $C$ by
\begin{equation}\label{eq:cubic-definition}
 C(X,Y,Z)=g(B(X,Y),JZ),
\end{equation}
for tangent vectors $X,Y,Z$ at the same point of $L$.
The tensor $C$ is fully symmetric;
see \cite[Lemma~3.10]{Oh1990}
and \cite[Lemma~3.1]{Oh1993}.
Let $(e_1,e_2)$ be a local positively oriented orthonormal frame,
and write $C_{ijk}=C(e_i,e_j,e_k)$ and $B_{ij}=B(e_i,e_j)$.
By minimality, $C$ is trace-free, that is,
\begin{equation}\label{eq:cubic-trace}
 \sum_{j=1}^{2} C(e_j,e_j,X)=0.
\end{equation}

In dimension two, symmetry and the trace-free condition
\eqref{eq:cubic-trace} determine $C$ from $c=C_{111}$ and $d=C_{112}$,
giving
\begin{equation}\label{eq:cubic-components}
 C_{111}=c,\qquad C_{112}=d,\qquad
 C_{122}=-c,\qquad C_{222}=-d.
\end{equation}
For a unit vector $v=\cos\theta\,e_1+\sin\theta\,e_2$, we have
\begin{equation}\label{eq:cubic-direction}
 C(v,v,v)=c\cos3\theta+d\sin3\theta.
\end{equation}
Taking the maximum over $\theta$ gives
\begin{equation}\label{eq:cubic-maximum}
 \max_{|v|=1}C(v,v,v)^2=c^2+d^2.
\end{equation}

Since $(Je_1,Je_2)$ is an orthonormal normal frame,
\[
 B_{11}=cJe_1+dJe_2,\qquad
 B_{12}=dJe_1-cJe_2,\qquad
 B_{22}=-cJe_1-dJe_2.
\]
Let $K$ denote the Gauss curvature of $L$.
Since the ambient sectional curvature of a Lagrangian tangent plane is $1$,
the Gauss equation gives
\begin{equation}\label{eq:gauss}
 K=1+g(B_{11},B_{22})-|B_{12}|^2=1-2(c^2+d^2).
\end{equation}

For background on Legendrian lifts and holonomy in the more general setting of embedded cyclic Lagrangian submanifolds, see \cite[Sections~3--4]{Oh1994}.
For local lifts and phase normalization in the minimal case,
see also \cite[Theorem~2.5 and Proposition~2.7]{DorfmeisterMa2021}.
We include the construction of the full normalized
three-sheeted covering needed below, which may be disconnected.

\begin{lemma}\label{lem:lift}
Let $f:L\to\CP^2$ be a minimal Lagrangian immersion of a closed oriented
surface, and let $\pi:\mathbb{S}^5\to\CP^2$ be the Hopf projection.
There exist a possibly disconnected three-sheeted covering
$\rho:\widehat L\to L$ and a horizontal minimal Legendrian immersion
$F:\widehat L\to \mathbb{S}^5$ satisfying $\pi\circ F=f\circ\rho$.

Equip $\widehat L$ with the orientation and metric pulled back from $L$.
For any local positively oriented orthonormal frame $(e_1,e_2)$ on
$\widehat L$, put $E_j=dF(e_j)$. Then
\begin{equation}\label{eq:normalized-frame}
 \mathcal F=(F,E_1,E_2)\in\SU(3),\qquad \Delta F=-2F.
\end{equation}
The surface $\widehat L$ is closed, and the covering admits a deck transformation $\tau$ satisfying
$F\circ\tau=\zeta F$, where $\zeta=e^{2\pi i/3}$.
If $f$ is an embedding, then $F$ is an embedding and its image meets each Hopf fiber over $f(L)$
in exactly three points.
\end{lemma}

\begin{proof}
Since $f^*\omega=0$, the pullback of the standard connection
on the Hopf bundle is flat.
Hence, on a sufficiently small contractible neighborhood $V\subset L$,
choose a horizontal lift $F:V\to \mathbb{S}^5$.
It is isometric because $\pi$ restricts to an isometry
on horizontal spaces. Since $F$ is horizontal and $\dim V=2$,
it is Legendrian.

For a positively oriented orthonormal frame $(e_1,e_2)$ on $V$, put
$E_j=dF(e_j)$. Differentiating $|F|^2=1$ and using horizontality
gives $\ip{F}{E_j}=0$.
Isometry and the Lagrangian condition give
\[
 \Re\ip{E_i}{E_j}=\delta_{ij},\qquad
 \Im\ip{E_i}{E_j}=f^*\omega(e_i,e_j)=0.
\]
Thus $\mathcal F=(F,E_1,E_2)$ is a unitary frame.

Let $\overline D$ be the Euclidean connection on $\C^3$.
To derive the frame equations, we decompose $\overline D_XE_j$
in the real orthonormal basis $F,iF,E_1,E_2,iE_1,iE_2$.
Differentiating the orthogonality relations with $F$ and $iF$ gives
\[
 \begin{aligned}
 \Re\ip{\overline D_XE_j}{F}
   &=-\Re\ip{E_j}{dF(X)}=-g(e_j,X),\\
 \Re\ip{\overline D_XE_j}{iF}
   &=-\Re\ip{E_j}{i\,dF(X)}
   =f^*\omega(e_j,X)=0.
 \end{aligned}
\]
The horizontal component of $\overline D_XE_j$ is the lift of
$df(\nabla_Xe_j)+B(X,e_j)$. Since $iE_\ell$ is the horizontal lift
of $Je_\ell$, the coefficient of $iE_\ell$ is
$g(B(X,e_j),Je_\ell)=C(X,e_j,e_\ell)$.
Combining these components, we obtain
\begin{equation}\label{eq:real-frame-equations}
 \begin{split}
 \overline D_XF&=dF(X),\\
 \overline D_XE_j
 &=-g(X,e_j)F+dF(\nabla_Xe_j)
       +i\sum_{\ell=1}^{2}C(X,e_j,e_\ell)E_\ell.
 \end{split}
\end{equation}

Let $\Theta=\det\nolimits_{\mathbb{C}}\mathcal F$. By Jacobi's formula,
\eqref{eq:real-frame-equations}, and \eqref{eq:cubic-trace},
\[
 \Theta^{-1}X(\Theta)
 =\tr(\mathcal F^*\overline D_X\mathcal F)
 =i\sum_{j=1}^2 C(X,e_j,e_j)=0.
\]
Thus $\Theta$ is constant on $V$ and has modulus one.
Multiplying $F$ by a constant $u\in \mathbb{S}^1$ preserves horizontality
and changes $\Theta$ to $u^3\Theta$.
Exactly three choices of $u$ satisfy $u^3\Theta=1$, and each gives
a normalized horizontal lift on $V$.

Taking the trace in \eqref{eq:real-frame-equations}
and using \eqref{eq:cubic-trace}, we obtain
\[
 \Delta F=-2F+i\sum_{\ell=1}^2
 \left(\sum_{j=1}^2 C(e_j,e_j,e_\ell)\right)E_\ell=-2F.
\]
Since $\Delta F=-2F+2H^{\mathbb{S}^5}$, where $H^{\mathbb{S}^5}$ is the
mean curvature vector in $\mathbb{S}^5$, the lift $F$ is minimal.

We now construct the covering globally. For $(x,z)\in f^*\mathbb{S}^5$,
choose a positively oriented orthonormal basis $(e_1,e_2)$ of
$T_xL$, and let $E_j$ be the horizontal lift of $df(e_j)$ at $z$.
Define
$\Theta(x,z)=\det\nolimits_{\mathbb{C}}(z,E_1,E_2)$.
A change of this basis acts on the last two columns by a matrix in
$\SO(2)$ and hence leaves the determinant unchanged.
Thus $\Theta$ is a well-defined smooth $\mathbb{S}^1$-valued function on
$f^*\mathbb{S}^5$. The horizontal lift of $df(e_j)$ at $uz$ is $uE_j$, so
\[
 \Theta(x,uz)=u^3\Theta(x,z),\qquad u\in \mathbb{S}^1.
\]
Set
\[
 \widehat L=\{(x,z)\in f^*\mathbb{S}^5:\Theta(x,z)=1\},
\]
and define $\rho:\widehat L\to L$ by $\rho(x,z)=x$.
Around each point of $L$, choose a connected neighborhood $V$
carrying a normalized horizontal lift $F_V$.
In the trivialization $(x,u)\mapsto(x,uF_V(x))$, where $u\in \mathbb{S}^1$,
the condition $\Theta=1$ is equivalent to $u^3=1$.
Thus $\rho^{-1}(V)$ is the disjoint union of the three smooth
graphs of $F_V,\zeta F_V,\zeta^2F_V$.
This proves that $\rho$ is a smooth three-sheeted covering.

Equip $\widehat L$ with the metric and orientation pulled back by $\rho$.
Define $F:\widehat L\to \mathbb{S}^5$ by $F(x,z)=z$, so that
$\pi\circ F=f\circ\rho$.
On each local sheet, $F$ agrees with one of the normalized horizontal
lifts above. Thus $F$ is isometric, Legendrian, and minimal, and
satisfies $\Delta F=-2F$.
Since $\Theta=1$ on $\widehat L$, we have
$\det\nolimits_{\mathbb C}(F,dF(e_1),dF(e_2))=1$
for every local positively oriented orthonormal frame $(e_1,e_2)$.
The map $\tau(x,z)=(x,\zeta z)$ is a deck transformation of order
three and satisfies $F\circ\tau=\zeta F$.

The pullback bundle is a closed subset of the compact space
$L\times \mathbb{S}^5$, and $\widehat L=\Theta^{-1}(1)$ is closed in that
bundle. Thus $\widehat L$ is compact. Its local sheets are
diffeomorphic to open subsets of $L$, so it has no boundary.

Finally, suppose $f$ is an embedding and
$F(x,z)=F(y,w)$. Then $z=w$, so
$f(x)=\pi(z)=\pi(w)=f(y)$. Injectivity of $f$ gives $x=y$,
and hence $(x,z)=(y,w)$. Thus $F$ is an injective immersion of a
compact manifold into $\mathbb{S}^5$, and is therefore an embedding.
For each $x\in L$, the three points of $\rho^{-1}(x)$ map to
three distinct points in the Hopf fiber over $f(x)$.
If $F(y,w)$ lies in the same fiber, then $f(y)=f(x)$ and hence $y=x$.
Thus the image of $F$ meets each Hopf fiber over $f(L)$
in exactly three points.
\end{proof}

For the rest of the paper, unless stated otherwise,
$L$ is closed and $f$ is an embedding.
Write $\Gamma=F(\widehat L)\subset \mathbb{S}^5$.
For $x,y\in L$, choose normalized local horizontal lifts
$F_1,F_2$ near $x,y$, respectively, and set
$a(x,y)=\ip{F_1(x)}{F_2(y)}$.
By Lemma~\ref{lem:lift}, changing either lift multiplies $a$
by a cube root of unity. Thus $|a|^2$ and $a^3$ are independent
of these choices and define smooth functions on $L\times L$,
as do $D$ and $G$ in \eqref{eq:intro-functions}.

The Cauchy--Schwarz inequality gives $|a|\leq1$, with equality
exactly when the two lifts differ by a unit complex scalar.
This is equivalent to $f(x)=f(y)$, and hence to $x=y$ because
$f$ is injective. On the diagonal, the two normalized lifts
differ by a cube root of unity, so $|a|^2=a^3=1$. Thus
\begin{equation}\label{eq:D-positive}
 D(x,y)>0\quad\text{if }x\ne y,\qquad D(x,x)=G(x,x)=0.
\end{equation}

\begin{lemma}\label{lem:diagonal}
The quotient $\Phi=G/D^2$, defined away from the diagonal of $L\times L$,
extends continuously to the diagonal with value $3/4$. For every
$x\in L$ and unit $v\in T_xL$,
\begin{equation}\label{eq:diagonal-expansion}
 \Phi(x,\exp_x(tv))
 =\frac34+\frac14\left(\frac12-C_x(v,v,v)^2\right)t^2+O(t^3),
\end{equation}
as $t\to0$, with the remainder uniform on the unit tangent
bundle of $L$.
\end{lemma}

\begin{proof}
Fix $x\in L$ and a unit vector $v\in T_xL$, and let
$\gamma(t)=\exp_x(tv)$.
Choose a normalized horizontal lift $F$ on a neighborhood of $x$.
Since $D$ and $G$ are independent of the choices of normalized
lifts, we may use $F$ at both $x$ and $\gamma(t)$ for
sufficiently small $t$.
Along $\gamma$, choose a parallel positively oriented orthonormal
frame $(e_1,e_2)$ with $e_1=\dot\gamma$.
Write $F(t)=F(\gamma(t))$ and
$E_j(t)=dF_{\gamma(t)}(e_j(t))$.
Set $c(t)=C(e_1,e_1,e_1)$ and $d(t)=C(e_1,e_1,e_2)$.
Since $\nabla_{\dot\gamma}e_j=0$,
\eqref{eq:real-frame-equations} and \eqref{eq:cubic-components}
give
\begin{equation}\label{eq:geodesic-frame}
 \begin{split}
 F'&=E_1,\\
 E_1'&=-F+icE_1+idE_2,\\
 E_2'&=idE_1-icE_2.
 \end{split}
\end{equation}

Let $P=F(0)$ and $a(t)=\ip{P}{F(t)}$.
By \eqref{eq:geodesic-frame} and unitarity of $(F,E_1,E_2)$ at $t=0$,
\[
 a(0)=1,\qquad a'(0)=0,\qquad
 a''(0)=-1,\qquad a'''(0)=-ic(0).
\]
The last identity follows from the $F$-component
$-icF$ of $F'''$. Taylor's formula gives
\[
 a(t)=1-\frac12t^2-\frac{i}{6}c(0)t^3+O(t^4).
\]
Writing $a=X+iY$, we obtain
\begin{equation}\label{eq:a-expansion}
 X=1-\frac12t^2+O(t^4),\qquad
 Y=-\frac16c(0)t^3+O(t^4),\qquad
 D=t^2+O(t^4).
\end{equation}
The expansion for $D$ follows from $D=1-X^2-Y^2$, since
$X^2=1-t^2+O(t^4)$ and $Y^2=O(t^6)$.

To expand $\Phi-3/4$, we use the identity
\begin{equation}\label{eq:diagonal-polynomial}
 \begin{split}
 4G-3D^2
 &=1-6|a|^2+8\Re(a^3)-3|a|^4\\
 &=(1-X)^3(1+3X)
   -6(1+4X+X^2)Y^2-3Y^4.
 \end{split}
\end{equation}
Substituting \eqref{eq:a-expansion} into
\eqref{eq:diagonal-polynomial} gives
\[
 4G-3D^2=\left(\frac12-c(0)^2\right)t^6+O(t^7).
\]
Since $D=t^2(1+O(t^2))$, for sufficiently small $t\ne0$,
\[
 \Phi(x,\gamma(t))-\frac34
 =\frac{4G-3D^2}{4D^2}
 =\frac14\left(\frac12-c(0)^2\right)t^2+O(t^3).
\]
As $c(0)=C_x(v,v,v)$, this proves \eqref{eq:diagonal-expansion}.

It remains to prove uniformity and continuity across the diagonal.
Smoothness of $C$ and compactness of the unit tangent bundle
give uniform bounds for $c,d$ and their required derivatives
along unit-speed geodesics. The system
\eqref{eq:geodesic-frame} then gives uniform bounds for the
frame derivatives needed in Taylor's formula. Thus all the
remainders above are uniform in $(x,v)$. In particular, there
are constants $\varepsilon>0$ and $A>0$ such that
\[
 \left|\Phi(x,\exp_x(tv))-\frac34\right|\leq A t^2
 \quad\text{whenever }|v|=1,\quad 0<|t|<\varepsilon.
\]
Since $L$ is compact, it has positive injectivity radius.
Every sufficiently close distinct pair $(x,y)$ therefore has
the form $y=\exp_x(tv)$ with
$t=\operatorname{dist}_L(x,y)>0$ and $|v|=1$.
The uniform estimate proves continuity across the entire diagonal.
\end{proof}

\section{Two-point inequalities}\label{sec:two-point}

We establish a Hessian identity at critical points of the
two-point function $\mathcal Z_\psi=\psi(a,\bar a)$, where $a$
is the Hermitian product of local normalized horizontal lifts
and $\psi$ is a real-valued $C^2$ function.
Applied to $h=\Re a$ on $\Gamma\times\Gamma$ and to
$Z_k=G-kD^2$ on $L\times L$, it yields
positivity of $G$ off the diagonal and the sharp quotient
estimate, respectively.

Subscripts $a$ and $\bar a$ denote Wirtinger derivatives of the outer
function $\psi$, evaluated at $a=a(x,y)$. Since $\psi$ is real-valued,
\[
 d\mathcal Z_\psi=2\Re(\psi_a\,da).
\]

\begin{lemma}\label{lem:contact}
Let $F_1,F_2$ be local horizontal lifts of oriented minimal
Lagrangian immersions into $\CP^2$, defined near points $p,q$,
respectively, and normalized as in \eqref{eq:normalized-frame}.
Equip the product of their domains with the product of the
induced metrics.

Set
\[
 a(x,y)=\ip{F_1(x)}{F_2(y)},\qquad
 \mathcal Z_\psi(x,y)
 =\psi\bigl(a(x,y),\overline{a(x,y)}\bigr),
\]
where $\psi$ is a real-valued $C^2$ function on a neighborhood
of $a(p,q)$. Suppose that, at $(p,q)$,
\[
 d\mathcal Z_\psi=0,\qquad
 D:=1-|a|^2>0,\qquad
 \xi:=\psi_a\ne0.
\]
Then there exist positively oriented orthonormal tangent bases
$e_1,e_2$ at $p$ and $v_1,v_2$ at $q$ such that, setting
\[
 W_1=(e_1,v_1),\qquad W_2=(e_2,-v_2),
\]
we have
\begin{equation}\label{eq:contact-trace}
 \begin{split}
 \mathcal L\mathcal Z_\psi
 &:={}\Hess\mathcal Z_\psi(W_1,W_1)
      +\Hess\mathcal Z_\psi(W_2,W_2)\\
 &=-4\Re(a\xi)-4\Re\left(\frac{\xi^2}{\bar\xi}\right)
   +8D\left[
       \psi_{a\bar a}
       -\Re\left(\psi_{aa}\frac{\bar\xi}{\xi}\right)
      \right].
 \end{split}
\end{equation}
All quantities in this identity are evaluated at $(p,q)$.
The notation $\mathcal L\mathcal Z_\psi$ denotes the displayed
pointwise Hessian sum.
\end{lemma}

\begin{proof}
Start with arbitrary positively oriented orthonormal tangent
bases $e_1,e_2$ at $p$ and $v_1,v_2$ at $q$. Write
\[
 P=F_1(p),\qquad Q=F_2(q),\qquad
 E_i=dF_1(e_i),\qquad V_j=dF_2(v_j).
\]
Both lifted frames belong to $\SU(3)$. Set
\begin{equation}\label{eq:relative-frame}
 U=(P,E_1,E_2)^*(Q,V_1,V_2)\in\SU(3).
\end{equation}
Index its rows and columns by $0,1,2$, so that $U_{00}=a$.
Differentiating $a$ in each factor gives
\begin{equation}\label{eq:a-first-derivatives}
 \begin{aligned}
 da((e_i,0))&=\ip{E_i}{Q}=U_{i0},\\
 da((0,v_j))&=\ip{P}{V_j}=U_{0j}.
 \end{aligned}
\end{equation}

Since $\xi\ne0$, choose the unit complex number
\(
 e^{i\beta}=\frac{i\bar\xi}{|\xi|}
\).
Then
\begin{equation}\label{eq:beta}
 \xi e^{i\beta}=i|\xi|,\qquad
 e^{2i\beta}=-\frac{\bar\xi}{\xi}.
\end{equation}
The critical-point condition $d\mathcal Z_\psi=0$ gives
\[
 \Re(\xi U_{i0})=\Re(\xi U_{0j})=0
 \qquad (i,j=1,2).
\]
Thus
\[
 \begin{aligned}
 \mathbf b_1&:=e^{-i\beta}(U_{10},U_{20})^T\in\R^2,\\
 \mathbf b_2&:=e^{-i\beta}(U_{01},U_{02})^T\in\R^2.
 \end{aligned}
\]
By unitarity,
\[
 |\mathbf b_1|^2=|\mathbf b_2|^2=1-|a|^2=D>0.
\]
We may therefore make an orientation-preserving rotation of
$(e_1,e_2)$ so that $\mathbf b_1=(\sqrt D,0)^T$, and independently
rotate $(v_1,v_2)$ so that $\mathbf b_2=(\sqrt D,0)^T$.
Retaining the same notation for the rotated bases and the
resulting matrix, we obtain
\[
 U_{10}=U_{01}=\sqrt D\,e^{i\beta},\qquad
 U_{20}=U_{02}=0.
\]
These rotations preserve the frame determinants, so
$U$ still belongs to $\SU(3)$.

Orthogonality of columns $0$ and $2$ gives
$\sqrt D\,e^{-i\beta}U_{12}=0$, hence $U_{12}=0$.
Similarly, orthogonality of rows $0$ and $2$ gives $U_{21}=0$.
Orthogonality of columns $0$ and $1$ then gives
\[
 \bar a\sqrt D\,e^{i\beta}
 +\sqrt D\,e^{-i\beta}U_{11}=0,
\]
so $U_{11}=-\bar a e^{2i\beta}$. The determinant of the upper
left $2\times2$ block is
\[
 a(-\bar a e^{2i\beta})-D e^{2i\beta}
 =-(|a|^2+D)e^{2i\beta}
 =-e^{2i\beta}.
\]
Since $\det U=1$, we obtain $U_{22}=-e^{-2i\beta}$.
Consequently,
\begin{equation}\label{eq:critical-matrix}
 U=
 \begin{pmatrix}
 a&\sqrt D\,e^{i\beta}&0\\
 \sqrt D\,e^{i\beta}&-\bar a e^{2i\beta}&0\\
 0&0&-e^{-2i\beta}
 \end{pmatrix}.
\end{equation}
This normal form is imposed only at $(p,q)$.

As shown in the proof of Lemma~\ref{lem:lift}, the horizontal
lifts $F_1,F_2$ are minimal in $\mathbb{S}^5$. Hence
$\Delta F_\alpha=-2F_\alpha$ for $\alpha=1,2$.
We compute the Laplacian in each factor separately.
For the mixed Hessian, we use
$\nabla_{(e_i,0)}(0,v_j)=0$ for vector fields lifted
from the respective factors. This gives
\begin{equation}\label{eq:a-second-derivatives}
 \begin{aligned}
 \Delta_xa&=\ip{\Delta F_1}{Q}=-2a,\\
 \Delta_ya&=\ip{P}{\Delta F_2}=-2a,\\
 \Hess a((e_i,0),(0,v_j))
   &=\ip{E_i}{V_j}=U_{ij}.
 \end{aligned}
\end{equation}

Using the bases chosen above, set
\[
 W_1=(e_1,v_1),\qquad W_2=(e_2,-v_2).
\]
Linearity of $da$ and bilinearity of $\Hess a$ give
\begin{equation}\label{eq:a-W}
 \begin{aligned}
 da(W_1)&=U_{10}+U_{01}=2\sqrt D\,e^{i\beta},\\
 da(W_2)&=U_{20}-U_{02}=0,\\
 \sum_{j=1}^2\Hess a(W_j,W_j)
   &=\Delta_xa+\Delta_ya+2U_{11}-2U_{22}\\
   &=-4a+2U_{11}-2U_{22}.
 \end{aligned}
\end{equation}

For any real tangent vector $W$ to the product at $(p,q)$,
the Hessian chain rule gives
\begin{equation}\label{eq:chain-rule}
 \Hess\mathcal Z_\psi(W,W)
 =2\Re\bigl(
    \xi\,\Hess a(W,W)+\psi_{aa}\,da(W)^2
   \bigr)
  +2\psi_{a\bar a}|da(W)|^2.
\end{equation}
Summing over $W_1,W_2$ and using \eqref{eq:a-W}, we obtain
\[
 \begin{aligned}
 \mathcal L\mathcal Z_\psi
 &=2\Re\bigl(\xi(-4a+2U_{11}-2U_{22})\bigr)\\
 &\quad
   +8D\bigl(\psi_{a\bar a}
             +\Re(\psi_{aa}e^{2i\beta})\bigr).
 \end{aligned}
\]

Equations \eqref{eq:beta} and \eqref{eq:critical-matrix} give
\[
 \xi U_{11}=\bar a\bar\xi,\qquad
 \xi U_{22}=\frac{\xi^2}{\bar\xi}.
\]
Hence
\[
 \begin{aligned}
 2\Re\bigl(\xi(-4a+2U_{11}-2U_{22})\bigr)
 &=-8\Re(a\xi)+4\Re(\bar a\bar\xi)
   -4\Re\left(\frac{\xi^2}{\bar\xi}\right)\\
 &=-4\Re(a\xi)
   -4\Re\left(\frac{\xi^2}{\bar\xi}\right).
 \end{aligned}
\]
Using $e^{2i\beta}=-\bar\xi/\xi$ in the remaining term
now yields \eqref{eq:contact-trace}.
\end{proof}

\begin{lemma}\label{lem:farthest}
Let $\Gamma\subset \mathbb{S}^5$ be a closed oriented minimal
Legendrian surface, possibly disconnected. Suppose that
$\det\nolimits_{\mathbb C}(P,E_1,E_2)=1$ for every
$P\in\Gamma$ and every positively oriented orthonormal
basis $(E_1,E_2)$ of $T_P\Gamma$.
Define $h:\Gamma\times\Gamma\to\mathbb R$ by
$h(P,Q)=\Re\ip{P}{Q}$.
If $(P,Q)$ is a global minimum point of $h$, then
$|\ip{P}{Q}|=1$.
\end{lemma}

\begin{proof}
Let $(P,Q)$ be a global minimum point of $h$, and put
$a=\ip{P}{Q}$. Suppose that $|a|<1$, so that
$D=1-|a|^2>0$.

The Hopf projection restricts to a minimal Lagrangian
immersion of $\Gamma$ into $\CP^2$. The local inclusion
maps near $P$ and $Q$ are its normalized horizontal lifts.
Set $\psi_h(a,\bar a)=(a+\bar a)/2$. Then
$\mathcal Z_{\psi_h}=h$, $(\psi_h)_a=1/2$, and
$(\psi_h)_{aa}=(\psi_h)_{a\bar a}=0$.
At the minimizing pair, $dh=0$ and $\Hess h$ is positive
semidefinite. Applying Lemma~\ref{lem:contact} with
$\psi=\psi_h$ gives
\[
 0\leq\mathcal Lh=-2(1+\Re a)<0,
\]
where the strict inequality follows from
$\Re a\geq-|a|>-1$.
This contradiction proves $|a|=1$.
\end{proof}

\begin{proposition}\label{prop:triangle}
Let $f:L\to\CP^2$ be a minimal Lagrangian embedding of a
closed oriented surface, and let $\Gamma=F(\widehat L)\subset \mathbb{S}^5$
be its full normalized lift given by Lemma~\ref{lem:lift}.
Then
\begin{equation}\label{eq:height-bound}
 h(P,Q)=\Re\ip{P}{Q}\geq-\frac12
 \qquad(P,Q\in\Gamma).
\end{equation}
Equality holds precisely when $Q=\zeta P$ or $Q=\zeta^2P$.
For distinct $x,y\in L$, choose any normalized lifts
$P,Q\in\Gamma$ over $x,y$, respectively, and set $a=\ip{P}{Q}$.
Then
\begin{equation}\label{eq:triangle-bound}
 1+2\Re(\zeta^j a)>0\quad(j=0,1,2),
 \qquad G(x,y)>0.
\end{equation}
\end{proposition}

\begin{proof}
By Lemma~\ref{lem:lift}, $\Gamma$ is compact and contains
$\zeta P,\zeta^2P$ whenever it contains $P$.
The function $h$ therefore attains a minimum $m$.
At a minimizing pair $(P,Q)$,
Lemma~\ref{lem:farthest} gives $|\ip{P}{Q}|=1$.
Equality in the Hermitian Cauchy--Schwarz inequality
implies that $P$ and $Q$ lie in the same Hopf fiber.
Since $f$ is an embedding, their base points in $L$ coincide.
Lemma~\ref{lem:lift} then gives
$Q\in\{P,\zeta P,\zeta^2P\}$, so $m\in\{1,-1/2\}$.
Since $h(P,\zeta P)=-1/2$, we have $m=-1/2$.
This proves \eqref{eq:height-bound}.

If equality holds in \eqref{eq:height-bound}, then
$h$ attains its global minimum at $(P,Q)$.
The same argument gives $Q\in\{P,\zeta P,\zeta^2P\}$.
The case $Q=P$ is excluded by $h(P,P)=1$.
Conversely, $Q=\zeta P$ and $Q=\zeta^2P$ both give equality,
since $\Re\zeta=\Re\zeta^2=-1/2$.

Now let $x\ne y$ and choose normalized lifts $P,Q$, with
$a=\ip{P}{Q}$. For $j=0,1,2$, the point $\zeta^jQ$ belongs
to $\Gamma$ and projects to $f(y)\ne f(x)$.
Thus $(P,\zeta^jQ)$ cannot be an equality pair in
\eqref{eq:height-bound}.
Since $\ip{P}{\zeta^jQ}=\zeta^j a$, we obtain
$1+2\Re(\zeta^j a)>0$ for $j=0,1,2$.

Changing the normalized lifts permutes these three factors,
so their product depends only on $x$ and $y$.
Writing $a=X+iY$ and using $\zeta=-1/2+i\sqrt3/2$,
we compute
\begin{equation}\label{eq:triangle-product}
 \begin{aligned}
 \prod_{j=0}^{2}\bigl(1+2\Re(\zeta^j a)\bigr)
  &=(1+2X)\bigl((1-X)^2-3Y^2\bigr)\\
  &=1-3(X^2+Y^2)+2(X^3-3XY^2)\\
  &=1-3|a|^2+2\Re(a^3)=G(x,y).
 \end{aligned}
\end{equation}
Every factor is strictly positive, so $G(x,y)>0$.
\end{proof}

Since $D>0$ off the diagonal,
Proposition~\ref{prop:triangle} gives $\Phi=G/D^2>0$ there.
To establish the sharp bound $\Phi\geq3/4$, we consider
$Z_k=G-kD^2$ for $0<k<3/4$.
The following lemma provides the algebraic input for applying
the Hessian identity \eqref{eq:contact-trace} to this family.

\begin{lemma}\label{lem:algebra}
Fix $0<k<3/4$ and define the real-valued polynomial
\begin{equation}\label{eq:barrier-polynomial}
 \psi_k(a,\bar a)=1-3|a|^2+a^3+\bar a^3-k(1-|a|^2)^2.
\end{equation}
Suppose $|a|<1$ and $\psi_k(a,\bar a)=0$. Put
\begin{equation}\label{eq:algebra-parameters}
 D=1-|a|^2,\qquad \eta=3-4k,\qquad
 M=\eta(3-k)+2Dk^2,\qquad \xi=(\psi_k)_a.
\end{equation}
Then $\xi\ne0$, and
\begin{equation}\label{eq:algebra-trace}
 \begin{split}
 &-4\Re(a\xi)-4\Re\frac{\xi^2}{\bar\xi}
   +8D\left((\psi_k)_{a\bar a}
        -\Re\left((\psi_k)_{aa}\frac{\bar\xi}{\xi}\right)\right)\\
 &\hspace{35mm}=-\frac{D\eta}{M}P_k(D)<0,
 \end{split}
\end{equation}
where
\begin{equation}\label{eq:Pk}
 P_k(D)=18-12k-16k^2+D(-30k+44k^2)-4D^2k^3.
\end{equation}
\end{lemma}

\begin{proof}
We use the constraint $\psi_k=0$ and clear denominators
to reduce the desired strict negativity to the positivity
of $P_k(D)$. To this end, write
$s=|a|^2=1-D$, $b=\Re(a^3)$, and $\mu=3-2kD$.
Differentiating $\psi_k$, we obtain
\begin{equation}\label{eq:psi-derivatives}
 \xi=3a^2-\mu\bar a,\qquad
 (\psi_k)_{aa}=6a-2k\bar a^2,\qquad
 (\psi_k)_{a\bar a}=-3+2k-4ks.
\end{equation}
The condition $\psi_k=1-3s+2b-kD^2=0$ gives
$b=1-\frac32D+\frac{k}{2}D^2$.

We first compute $N=|\xi|^2$ and $R=\Re(a\xi)$ in terms of $D$ and $k$.
Direct multiplication gives
\[
 \begin{aligned}
 N&=9s^2+\mu^2s-6\mu b\\
  &=9(1-D)^2+(3-2kD)^2(1-D)-3(3-2kD)(2-3D+kD^2)\\
  &=D^2(9-15k+4k^2+2Dk^2)=D^2M,
 \end{aligned}
\]
and
\[
 \begin{aligned}
 R&=3b-\mu s\\
  &=\frac32(2-3D+kD^2)-(3-2kD)(1-D)\\
  &=-\frac D2(\eta+kD).
 \end{aligned}
\]
Thus
\begin{equation}\label{eq:NR-values}
 N=D^2M>0,\qquad R=-\frac D2(\eta+kD).
\end{equation}
Here $0<D\leq1$, $\eta=3-4k>0$, and $3-k>0$, so
$M=\eta(3-k)+2Dk^2>0$. In particular, $\xi\ne0$.

Denote the left-hand side of \eqref{eq:algebra-trace} by
$\mathcal E$, and set
\[
 X=\Re(\xi^3),\qquad
 Y=\Re\bigl((\psi_k)_{aa}\bar\xi^2\bigr).
\]
Since $N=\xi\bar\xi$, we have
\[
 \frac{\xi^2}{\bar\xi}=\frac{\xi^3}{N},\qquad
 (\psi_k)_{aa}\frac{\bar\xi}{\xi}
   =\frac{(\psi_k)_{aa}\bar\xi^2}{N}.
\]
Consequently,
\begin{equation}\label{eq:cleared-algebra-trace}
 \frac{N\mathcal E}{4}=-RN-X+2D\bigl((\psi_k)_{a\bar a}N-Y\bigr).
\end{equation}

Using $\Re(a^6)=2b^2-s^3$, we expand
$(3a^2-\mu\bar a)^3$ and
$(6a-2k\bar a^2)(3\bar a^2-\mu a)^2$ to obtain
\begin{equation}\label{eq:direct-complex-numerators}
 \begin{aligned}
 X&=27(2b^2-s^3)-(27\mu s+\mu^3)b+9\mu^2s^2,\\
 Y&=(54s+6\mu^2+12k\mu s)b-(36\mu+2k\mu^2)s^2\\
  &\quad-18k(2b^2-s^3).
 \end{aligned}
\end{equation}
Substituting $s=1-D$, $b=1-\frac32D+\frac{k}{2}D^2$, and
$\mu=3-2kD$ into \eqref{eq:direct-complex-numerators}, and using
$(\psi_k)_{a\bar a}=-3-2k+4kD$ and $N=D^2M$,
we collect powers of $D$ to obtain
\begin{equation}\label{eq:direct-numerator-polynomials}
 \begin{aligned}
 X&=\frac{D^3}{2}\Bigl[54-135k+72k^2+16k^3\\
  &\hspace{24mm}+D(45k^2-60k^3)+8D^2k^4\Bigr],\\
 (\psi_k)_{a\bar a}N-Y&=D^3k\bigl(9-18k+8k^2+Dk^2\bigr).
 \end{aligned}
\end{equation}
The coefficients factor as
\[
 \begin{aligned}
 54-135k+72k^2+16k^3&=(6+k)\eta^2,\\
 45k^2-60k^3&=15k^2\eta,\\
 9-18k+8k^2&=(3-2k)\eta.
 \end{aligned}
\]
Consequently,
\[
 \begin{aligned}
 X&=\frac{D^3}{2}
   \bigl((6+k)\eta^2+15D\eta k^2+8D^2k^4\bigr),\\
 (\psi_k)_{a\bar a}N-Y&=D^3k\bigl((3-2k)\eta+Dk^2\bigr).
 \end{aligned}
\]
Using $-RN=\frac{D^3}{2}(\eta+kD)M$ and
$M=\eta(3-k)+2Dk^2$, we substitute the preceding expressions into
\eqref{eq:cleared-algebra-trace} to obtain
\begin{equation}\label{eq:combine-polynomial}
 \begin{aligned}
 \frac{N\mathcal E}{4}
 &=\frac{D^3\eta}{2}
   \bigl(-(3+2k)\eta+(15k-22k^2)D+2k^3D^2\bigr)\\
 &=-\frac{D^3\eta}{4}P_k(D).
 \end{aligned}
\end{equation}
Here the last equality follows from
$2(3+2k)\eta=18-12k-16k^2$ and the definition of $P_k(D)$
in \eqref{eq:Pk}.
Since $N=D^2M>0$, it follows that
$\mathcal E=-D\eta P_k(D)/M$.

Finally, $P_k''(D)=-8k^3<0$, so $P_k$ is concave on $[0,1]$.
Its endpoint values satisfy
\[
 P_k(0)=\eta(6+4k)>0,\qquad
 P_k(1)=(1-k)(6\eta+4k^2)>0.
\]
Hence, for $0<D\leq1$,
\begin{equation}\label{eq:Pk-positive}
 P_k(D)\geq(1-D)P_k(0)+DP_k(1)>0.
\end{equation}
Since $D,\eta,M>0$, it follows that $\mathcal E<0$, as required.
\end{proof}

\begin{proposition}\label{prop:sharp}
Let $f:L\to\CP^2$ be a minimal Lagrangian embedding of a
closed oriented surface.
For distinct $x,y\in L$, choose normalized local horizontal
lifts $F_1,F_2$ of $f$ near $x,y$, respectively, and set
$a=\ip{F_1(x)}{F_2(y)}$.
Then
\begin{equation}\label{eq:sharp-quotient}
 \Phi(x,y)=\frac{G(x,y)}{D(x,y)^2}
 =\frac{1-3|a|^2+2\Re(a^3)}{(1-|a|^2)^2}
 \geq\frac34.
\end{equation}
\end{proposition}

\begin{proof}
Lemma~\ref{lem:diagonal} gives a continuous extension of $\Phi$
to $L\times L$, with value $3/4$ on the diagonal.
This extension, still denoted by $\Phi$, is positive by
Proposition~\ref{prop:triangle}. Compactness therefore gives
$0<k:=\min_{L\times L}\Phi\leq3/4$.

Suppose, for contradiction, that $k<3/4$, and choose a
minimizing pair $(x_0,y_0)$. Since $\Phi=3/4$ on the diagonal,
we have $x_0\ne y_0$.

Consider the globally smooth function
\begin{equation}\label{eq:global-Z}
 Z_k=G-kD^2.
\end{equation}
Off the diagonal, $Z_k=D^2(\Phi-k)\geq0$.
On the diagonal, $D=G=0$, so $Z_k=0$ there as well.
Thus $Z_k$ attains its minimum value zero at $(x_0,y_0)$,
where $dZ_k=0$ and $\Hess Z_k$ is positive semidefinite.

Choose normalized local horizontal lifts near $x_0$ and $y_0$,
and let $a$ denote their Hermitian product. Then
$Z_k=\psi_k(a,\bar a)=\mathcal Z_{\psi_k}$ locally, with
$\psi_k$ as in \eqref{eq:barrier-polynomial}.
At $(x_0,y_0)$, we have $|a|<1$ and
$\psi_k(a,\bar a)=0$, since $D>0$ and $Z_k=0$ there.
Lemma~\ref{lem:algebra} therefore gives
$\xi=(\psi_k)_a\ne0$.
Since $dZ_k=0$ there, Lemma~\ref{lem:contact} applies.

Let $W_1,W_2$ be the vectors supplied by Lemma~\ref{lem:contact}.
At $(x_0,y_0)$, the minimum condition,
\eqref{eq:contact-trace}, and Lemma~\ref{lem:algebra} give
\begin{equation}\label{eq:negative-trace}
 \begin{aligned}
 0\leq\mathcal LZ_k
 &=\Hess Z_k(W_1,W_1)+\Hess Z_k(W_2,W_2)\\
 &=-\frac{D\eta}{M}P_k(D)<0,
 \end{aligned}
\end{equation}
where $\eta$ and $M$ are defined in
\eqref{eq:algebra-parameters}.
This contradiction proves \eqref{eq:sharp-quotient}.
\end{proof}

\begin{remark}[The Clifford model and the sharp constant]\label{rem:clifford-discriminant}
For the Clifford torus, the quantity $4G-3D^2$
admits an explicit factorization that characterizes
equality in \eqref{eq:sharp-quotient}.
On the universal cover $\R^2$ of the Clifford torus, consider the normalized horizontal lift
\begin{equation}\label{eq:model-lift}
 F(s,t)=-\frac1{\sqrt3}
 \left(
 e^{i\sqrt2s},
 e^{-is/\sqrt2+i\sqrt{3/2}\,t},
 e^{-is/\sqrt2-i\sqrt{3/2}\,t}
 \right).
\end{equation}
The induced metric is $ds^2+dt^2$, and $\ip{F}{dF}=0$.
With the standard orientation of $\R^2$, the frame
$(F,F_s,F_t)$ belongs to $\SU(3)$.
Every normalized lift of a point of the Clifford torus has the form
$\zeta^mF(s,t)$ for some $m\in \{0,1,2\}$, and its coordinate product is $-1/(3\sqrt3)$.

Let $P,Q$ be normalized lifts of two distinct points of the Clifford torus and put $\lambda_j=Q_j/P_j$ for $j=0,1,2$.
Then
\[
 |\lambda_j|=1,\qquad
  \lambda_0\lambda_1\lambda_2=1,\qquad
 a=\ip{P}{Q}=\frac{\lambda_0+\lambda_1+\lambda_2}{3}.
\]
Since $\lambda_j^{-1}=\bar\lambda_j$, the pairwise products satisfy
\[
 \lambda_0\lambda_1+\lambda_0\lambda_2+\lambda_1\lambda_2
   =\bar\lambda_2+\bar\lambda_1+\bar\lambda_0=3\bar a.
\]
Thus the $\lambda_j$ are the roots of
$p(z)=z^3-3az^2+3\bar a z-1$, whose discriminant is
\[
 \begin{aligned}
  \operatorname{disc}p
    &=\prod_{i<j}(\lambda_i-\lambda_j)^2
    =-27(1-6|a|^2+8\Re(a^3)-3|a|^4)\\
    &=-27(4G-3D^2).
 \end{aligned}
\]

For roots on the unit circle,
\[
 |\lambda_i-\lambda_j|^2
   =(\lambda_i-\lambda_j)
      \left(\frac1{\lambda_i}-\frac1{\lambda_j}\right)
   =-\frac{(\lambda_i-\lambda_j)^2}{\lambda_i\lambda_j}.
\]
Multiplying over the three pairs $i<j$ and using
$\lambda_0\lambda_1\lambda_2=1$, we obtain
\begin{equation}\label{eq:clifford-discriminant}
 27(4G-3D^2)=\prod_{i<j}|\lambda_i-\lambda_j|^2.
\end{equation}
Since $D>0$, this can be written as
\[
 \Phi-\frac34
    =\frac{1}{108D^2}\prod_{i<j}|\lambda_i-\lambda_j|^2\geq0.
\]
Equality holds precisely when two of the $\lambda_j$
coincide. The three ratios cannot all coincide, since
this would give $Q=\lambda P$, contrary to the assumption
that the underlying points are distinct.
For example, the normalized lifts 
$P=-\frac{1}{\sqrt{3}}(1,1,1)$ and $Q=\frac{1}{\sqrt{3}}(1,1,-1)$ give 
$(\lambda_0,\lambda_1,\lambda_2)=(-1,-1,1)$ and hence
\[
 a=-\frac13,\qquad D=\frac89,\qquad
 G=\frac{16}{27},\qquad \Phi=\frac34.
\]
Their projections are distinct points
$[1:1:1]$ and $[-1:-1:1]$ of the Clifford torus.
Thus equality in \eqref{eq:sharp-quotient} is attained
by distinct points of the Clifford torus.
\end{remark}

\section{Proof of Theorem \ref{thm:main}}\label{sec:rigidity}
We now combine the sharp two-point inequality with the diagonal
expansion to obtain nonnegative Gauss curvature and complete
the proof of Theorem~\ref{thm:main}.

\begin{lemma}\label{lem:nonnegative-curvature}
Every closed oriented embedded minimal Lagrangian surface in $\CP^2$
satisfies
\[
 C(v,v,v)^2\leq\frac12\qquad (|v|=1),\qquad K\geq0.
\]
\end{lemma}

\begin{proof}
Fix $x\in L$ and a unit vector $v\in T_xL$. For every sufficiently
small nonzero $t$, the point $\exp_x(tv)$ is distinct from $x$.
The global estimate in Proposition~\ref{prop:sharp} therefore gives
\[
 \Phi(x,\exp_x(tv))-\frac34\geq0.
\]
By \eqref{eq:diagonal-expansion}, 
\[
 0\leq\frac{\Phi(x,\exp_x(tv))-3/4}{t^2}
   =\frac14\left(\frac12-C_x(v,v,v)^2\right)+O(t).
\]
Letting $t\to0$ gives $C_x(v,v,v)^2\leq1/2$ for every $x\in L$ and every unit vector $v\in T_xL$. Taking the maximum over $v$ and using
\eqref{eq:cubic-maximum} yields $c^2+d^2\leq1/2$.
The Gauss equation \eqref{eq:gauss} then gives
$K=1-2(c^2+d^2)\geq0$.
\end{proof}

The following classical classification is due to Yau \cite{Yau1974} and Naitoh--Takeuchi
\cite{NaitohTakeuchi1982}; see
\cite[Theorem~1(b)]{CastroUrbano1994}.

\begin{proposition}\label{prop:flat-classification}
Let $f:T^2\to\CP^2$ be a minimal Lagrangian immersion whose induced
metric is flat. Then its image is congruent to $\TCl$.
\end{proposition}

\begin{proof}[Proof of Theorem~\ref{thm:main}]
Choose an orientation of the given torus $T^2$.
Lemma~\ref{lem:nonnegative-curvature} gives $K\geq0$. The Gauss--Bonnet theorem and $\chi(T^2)=0$ imply
\[
 \int_{T^2}K\,dA=2\pi\chi(T^2)=0.
\]
Hence $K\equiv 0$, and the induced metric is flat.
Proposition~\ref{prop:flat-classification} now supplies an element
$A\in\PU(3)$ with $f(T^2)=A(\TCl)$, as asserted.
\end{proof}

\par\vspace{2mm}
\noindent\textbf{Acknowledgements.}
We thank Professor Ben Andrews for valuable discussions with the second author
in 2012 on a two-point function approach to this problem.
The first author thanks Professor Guofang Wang for introducing
him to this problem and providing guidance during the first
author's doctoral studies.
We also thank Professor Yong-Geun Oh for his interest in this work and his helpful comments.
This work was supported by the National Natural Science Foundation
of China (Grant Nos. 12671062, 12201138, 12471048, W2521103)
and by the Natural Science Foundation of Henan Province
(Grant No. 262300421869).

\vskip 2mm
\noindent\textbf{AI Disclosure.}
The authors developed and independently verified the final proof.
ChatGPT assisted in the search for the two-point function
$Z_k=G-kD^2$ in \eqref{eq:global-Z}, and AI-assisted tools
were used for additional checks of calculations and exposition.
The authors take full responsibility for the mathematical
content and references.

\end{document}